\documentclass{amsart}

\usepackage[colorlinks=true, urlcolor=blue,bookmarks=true,bookmarksopen=true,citecolor=blue,hypertex]{hyperref}
\usepackage{graphicx}
\usepackage{enumerate}
\usepackage{parskip}

\usepackage{amscd}
\usepackage{amssymb}
\usepackage{amsxtra}
\usepackage{amsmath}
\usepackage{enumerate}
\usepackage{mathrsfs}

\usepackage{color}

\usepackage{amsmath,amsfonts,amssymb}
\usepackage{verbatim}
\usepackage[hmargin = 4cm,vmargin = 2.5cm]{geometry}
\usepackage{epsfig}
\usepackage{amsthm}
\usepackage{bbold}

\usepackage[T1]{fontenc}

\theoremstyle{plain}

\newtheorem{thm}{Theorem}

\newtheorem{lem}[thm]{Lemma}

\theoremstyle{definition}

\theoremstyle{remark}
\newtheorem{rem}[thm]{Remark}

\newtheorem*{qsnonum}{Question}

\theoremstyle{plain}

\newcommand{\Id}{{{\mathchoice {\rm 1\mskip-4mu l} {\rm 1\mskip-4mu l}
      {\rm 1\mskip-4.5mu l} {\rm 1\mskip-5mu l}}}}
\def\NN{\mathbb{N}}
\def\ZZ{\mathbb{Z}}

\def\RR{\mathbb{R}}

\def\Id{\mathbb{1}}

\def\dd{\mathrm{d}}
\DeclareMathOperator{\Ham}{Ham}

\DeclareMathOperator{\Diff}{Diff}
\DeclareMathOperator{\Cal}{Cal}
\DeclareMathOperator{\Vect}{Vect}

\DeclareMathOperator{\Area}{Area}

\definecolor{lgray}{rgb}{0.5,0.5,0.5}

\begin{document}

\pagestyle{headings}

\bibliographystyle{alphanum}

\title[$C^0$-gap between entropy-zero Hamiltonians and autonomous diffeomorphisms]{$C^0$-gap between entropy-zero Hamiltonians and autonomous diffeomorphisms of surfaces}

\date{\today} 

\author{Michael Brandenbursky}
\address{Michael Brandenbursky, Department of Mathematics,
Ben Gurion University of the Negev, Beer Sheva, Israel}
\email{brandens@bgu.ac.il}

\author{Michael Khanevsky}
\address{Michael Khanevsky, Faculty of Mathematics,
Technion - Israel Institute of Technology
Haifa, Israel}
\email{khanev@technion.ac.il}

\begin{abstract}
Let $\Sigma$ be a surface equipped with an area form. There is an long standing open question by Katok,
which, in particular, asks whether every entropy-zero Hamiltonian diffeomorphism of a surface 
lies in the $C^0$-closure of the set of integrable diffeomorphisms.
A slightly weaker version of this question asks: ``Does every entropy-zero Hamiltonian 
diffeomorphism of a surface lie in the $C^0$-closure of the set of autonomous diffeomorphisms?''

In this paper we answer in negative the later question. In particular, 
we show that on a surface $\Sigma$ the set of autonomous Hamiltonian diffeomorphisms is not $C^0$-dense
in the set of entropy-zero Hamiltonians. We explicitly construct examples of such Hamiltonians which cannot be approximated
by autonomous diffeomorphisms.
\end{abstract}

\maketitle

\section{Introduction}

There is the following long standing open question of Katok from seventies:
"\emph{In low dimensions is every volume-preserving dynamical system with zero topological entropy a limit of integrable systems?}"
This is stated as Problem 1 in ~\cite{Katok}, but relates also to the much older paper ~\cite{An-Kat}. 
Low dimensions means maps on surfaces or flows on $3$-dimensional manifolds. 
 
Let $\Sigma$ be a surface with an area form $\omega$ and
denote by $\Ham(\Sigma)$ the group of Hamiltonian diffeomorphisms of $(\Sigma, \omega)$ that are compactly supported in the interior of $\Sigma$. For every smooth function
$H\colon\Sigma\to\RR$ compactly supported in the interior, there exists a unique vector field $X_H$ which satisfies
$$
dH(\cdot)=\omega(X_H,\cdot).
$$
It is easy to see that $X_H$ is tangent to the level sets of $H$. Let $h$ be the
time-one map of the flow $h_t$ generated by $X_H$. The diffeomorphism $h$ is
area-preserving, it belongs to $\Ham(\Sigma)$ and every Hamiltonian diffeomorphism arising this way is called
{\em autonomous}. Such a diffeomorphism is easy to understand
in terms of its generating function.  More generally, autonomous maps $h \in \Diff (\Sigma)$ appear as time-$t$ maps
of time-independent flows and also have relatively simple dynamics.

We consider the following version of Katok's 
\begin{qsnonum}
Does every entropy-zero Hamiltonian of $\Sigma$ lie in the $C^0$-closure of the set of autonomous diffeomorphisms?
\end{qsnonum}

There was a very little progress on both questions till an appearance of a beautiful paper of Bramham ~\cite{Bramham}, where he partially solved
them in the case of the disc. More precisely, he showed that every irrational pseudo-rotation of the disc lies in the $C^0$-closure of the set of 
integrable autonomous diffeomorphisms (which are not necessarily area-preserving).

In this paper we answer in negative the above question for every surface $\Sigma$. Our main result is the following

\begin{thm} \label{T:main}
There exists an entropy-zero $g \in \Ham(\Sigma)$ which cannot be presented as a $C^0$-limit of autonomous diffeomorphisms.
\end{thm}

Diffeomorphism $g$ is constructed as a composition of two autonomous Hamiltonian diffeomorphisms supported in an annulus,
which is symplectically embedded in the surface $\Sigma$. In case when $\Sigma$ is an annulus $S^1 \times [0, 1]$, we 
present an additional argument which is based on the second author result from ~\cite{Kh:non-auton-curves}. It shows

\begin{thm} \label{T:annulus}
There exists an entropy-zero $g \in \Ham(S^1 \times [0, 1])$ which cannot be presented as a $C^0$-limit of autonomous Hamiltonians from 
$\Ham_{Aut} (S^1 \times [0, 1])$. Moreover, given $R>0$, one may ensure that the Hofer distance $d_H (g, \Ham_{Aut}) > R$.
\end{thm}

The proof uses the theory of Calabi quasimorphisms due to Entov-Polterovich ~\cite{En-Po:calqm}.
At the same time the general case does not require techniques from symplectic geometry.

We finish the introduction by noting that the constructed diffeomorphism $g$ is integrable, 
hence the original question of Katok remains open. 

\subsection*{Acknowledgments.}
MB~was partially supported by Humboldt research fellowship.

\section{Integrability, laminations and entropy}\label{S:Tools}
 
%

A diffeomorphism $g \in \Diff (\Sigma)$ is \emph{integrable} if it admits an invariant function $H : \Sigma \to \RR$
(a \emph{first integral}) such that $H$ is not constant on any open set.
Such a function $H$ induces a lamination on $\Sigma$ by its level sets. This lamination is regular away from the critical points of $H$ and 
is invariant under $g$.

If $g$ is an autonomous Hamiltonian generated by $H : \Sigma \to \RR$, the function $H$ serves as an integral for $g$. If $H$ is constant in certain
open sets (for example, outside its support) one may perturb $H$ there. Such a perturbed function is still invariant under $g$ since the perturbation occurs 
in the set of fixed points of $g$. Not all integrable Hamiltonians are autonomous, an example is the diffeomorphism $g$ constructed in the proof of
Theorem ~\ref{T:main}.
Note that a non area-preserving autonomous diffeomorphism $g \in \Diff (\Sigma)$ might be non-integrable, yet it admits an invariant 
lamination by flow lines.

Integrable maps on surfaces have zero entropy: indeed, according to ~\cite{Ka:Lyap-ent}, positive entropy implies existence of transverse 
homoclinic points which cannot show up in the integrable setting.
Zero-entropy does not imply integrability, a counterexample is provided by pseudo-rotations on a disk. 


\section{Proof of Theorem~\ref{T:annulus}} \label{S:Annulus}

\subsection{Quasimorphisms}
Let $G$ be a group. A function $r : G \to \RR$ is called a \emph{quasimorphism} if there exists 
a constant $\Delta$ (called the \emph{defect} of $r$) such that $|r(fg) - r(f) - r(g)| < \Delta$ for all $f, g \in G$. The quasimorphism $r$ 
is called \emph{homogeneous} if it satisfies $r(g^m) = mr(g)$ for all $g \in G$ and $m \in \ZZ$.
Any homogeneous quasimorphism satisfies $r(fg) = r(f) + r(g)$ for commuting elements $f, g$. Every quasimorphism
is equivalent (up to a bounded deformation) to a unique homogeneous one ~\cite{Ca:scl}.

\subsection{Calabi invariant and Calabi quasimorphisms} \label{SS:Calabi}

Let $F_t : \Sigma \to \RR$, $t \in [0, 1]$ be a time-dependent smooth function with compact support in the interior of $\Sigma$. We define
$\widetilde{\Cal} (F_t) = \int_0^1 \left( \int_\Sigma F_t \omega \right) \dd t$. If the symplectic form $\omega$ is exact 
(this is the case for an annulus or a disk),
$\widetilde{\Cal}$ descends to a homomorphism $\Cal_\Sigma: \Ham(\Sigma) \to \RR$ which is called 
the Calabi homomorphism.

\medskip

The proposed proof of Theorem~\ref{T:annulus} uses the \emph{Calabi quasimorphism} by Entov-Polterovich ~\cite{En-Po:calqm}. We start with a brief recollection of 
the relevant facts from that paper.

Let $\Sigma = S^1 \times [0, 1]$ be an annulus (we assume $S^1 = \RR / 2\pi \ZZ$ and pick a symplectic form 
$\omega = \frac{1}{2 \pi} \dd \theta \wedge \dd s$ so that $\Area (\Sigma) = 1$).
Given a compactly supported smooth function $F: \Sigma \to \RR$, the \emph{Reeb graph} $T_F$ is defined as the set of connected
components of level sets of $F$ (for a more detailed description we refer the reader to ~\cite{En-Po:calqm}). 
For a generic Morse function $F$ (saying `Morse', we mean that the restriction of $F$ to the interior of its support 
is a Morse function) this set, equipped with topology 
induced by the projection $\pi_F: \Sigma \to T_F$, is homeomorphic to 
a tri-valent tree. We endow $T_F$ with a positive measure given by $\mu (X) = \int_{\pi_F^{-1}(X)} \omega$ 
for all $X \subseteq T_F$ with measurable $\pi_F^{-1}(X)$. 
In the case of the annulus $\Sigma = S^1 \times [0, 1]$, $\pi_F (S^1 \times \{0\})$ will be 
referred to as the \emph{bottom root} of $T_F$ and $\pi_F (S^1 \times \{1\})$ as the \emph{top root}. The shortest path connecting the roots of $T_F$ will be called a \emph{stem}.

A point $x_{m} \in T_F$ is a \emph{median} of $T_F$ if all connected components of $T_F \setminus \{x_m\}$ have measure at most $\frac{\Area (\Sigma)}{2}$.
A median always exists and is unique (see ~\cite{En-Po:calqm}). The set $\pi_F^{-1} (x_m)$ will be called the \emph{median} with respect to $F$.
Suppose $\Sigma = S^1 \times [0, 1]$, we define \emph{percentile} sets in analogy to the median.
Let $h \in [0, 1]$. $x_h \in T_F$ is an \emph{$h$-percentile} of $T_F$ if the top and the bottom roots belong to different connected components of $T_F \setminus \{x_h\}$ 
and the connected component of the bottom root has measure $h \cdot \Area (\Sigma)$. 
The set $\pi_F^{-1} (x_h)$ is an $h$-percentile with respect to $F$. 

Clearly, percentiles correspond to points $x$ in the stem of $T_F$ and the percentile value increases monotonically 
along the stem. Unlike the median, if $T_F$ is not homeomorphic to an interval (that is, has `branches' besides the stem), 
$h$-percentiles do not exist for certain $h \in [0, 1]$. Each branch corresponds to a `gap' (missing interval) in the set of percentile values. Length of the gap is given by the 
measure of the branch normalized by $\Area (\Sigma)$. If an $h$-percentile exists, it is unique. 
The $\frac{1}{2}$-percentile (if it exists) coincides with the median. For a generic $F$ this corresponds to the case when the median set of $F$ is a non-contractible circle.
Using a standard Morse-theoretic argument, we conclude with the following observation: percentile sets are not contractible in $S^1 \times [0, 1]$. 
The set $A_F$ of points that are not percentiles of $T_F$
is the union of branches that grow out of the stem of $T_F$. The set $\pi^{-1} (A_F)$ is the union of topological disks corresponding to these branches.

\medskip

In ~\cite{En-Po:calqm} the authors describe construction of a homogeneous quasimorphism $$\Cal_{S^2} : \Ham (S^2) \to \RR.$$
It has the following properties: $\Cal_{S^2}$ is Hofer-Lipschitz $$|\Cal_{S^2}(\phi)| \leq \Area (S^2) \cdot \| \phi \|_H.$$
In the case when $\phi \in \Ham(S^2)$ is supported in a disk $D$ which is displaceable 
in $S^2$, $\Cal_{S^2} (\phi) = Cal_D (\phi \big|_D)$.
Moreover, for a $\phi \in \Ham(S^2)$ generated by an autonomous function $F: S^2 \to \RR$, 
$\Cal_{S^2} (\phi)$ can be computed in the following way. 
Let $x$ be the median of $T_F$ and $X = \pi_F^{-1} (x)$ be the corresponding subset of $S^2$.
Then 
\[
	\Cal_{S^2}(\phi) = \int_{S^2} F \omega - \Area (S^2) \cdot F(X).
\]

We embed $\Sigma$ into a sphere $S^2_{a,b}$ of area $1 + a + b$ by gluing a disk of area $a$ to $S^1 \times \{0\}$ and a disk of area $b$ to $S^1 \times \{1\}$.
Denote this embedding by $i_{a, b} : \Sigma \to S^2_{a,b}$.
Let $$r_{a,b} = \frac{1}{1+a+b} \cdot \left( \Cal_\Sigma - i_{a,b}^* \Cal_{S^2_{a,b}} \right)$$ be the normalized difference between the Calabi homomorphism on $\Sigma$ and the pullback of the Calabi quasimorphism of $S^2_{a,b}$.
Note that $r_{a,b}$ vanishes on Hamiltonians $g$ supported in a disk $D \subset \Sigma$ of area $\frac{1+a+b}{2}$. Indeed, $i_{a,b} (D)$ is displaceable in $S^2_{a,b}$ thus 
$$\Cal_{S^2_{a,b}} (i_{a,b,*} g) = \Cal_D \left(g \big|_D \right) = \Cal_\Sigma (g).$$ This implies that $r_{a,b}$ is continuous in the $C^0$-topology (see ~\cite{En-Po-Py:qm-continuity}).

Let $F: \Sigma \to \RR$ be a Hamiltonian function, $f$ its time-$1$ map and suppose that $-1 \leq b-a \leq 1$ or, equivalently, $h: = \frac{1+b-a}{2} \in [0, 1]$. If $F$ admits  
the $h$-percentile set $X_h$, it is mapped by $i_{a, b}$ to the median set of $i_{a, b, *} F : S^2_{a,b} \to \RR$, therefore 
$r_{a,b} (f) = F(X_h)$. This makes the quasimorphisms $r_{a,b}$ a useful tool to extract information about the Reeb graph of a 
Hamiltonian function.

\subsection{Construction of the Hamiltonian $g$}
Let $F : S^1 \times [0, 1] \to \RR$ be a Hamiltonian function
given by $K (\theta, s) = s$ when $s \in [0.01, 0.99]$ and extended to the rest of $\Sigma$ in arbitrary way. The time-$t$ map $\phi^t$ of $K$ rotates the annulus $A = S^1 \times [0.01, 0.99]$
by $t$ in the $S^1$ coordinate. Let $D \subset A$ be a disk of area $0.8$ and $\Psi : \Sigma \to \RR$ be a smooth function which equals $1$ in $D$ and is supported in a disk of area $0.9$ inside $A$. 
(That is, $\Psi$ is a smooth approximation of the indicator function of $D$.)
The time-$t$ map $\psi^t$ fixes $D$ pointwise but the flow induces a fast rotation outside $\partial D$.
Pick large independent parameters $T, \tau \in \NN$ and consider $g_{T, \tau} := \phi^T \circ \psi^\tau$. Assuming $T$ is an integer, $\phi^T$ translates 
the subannulus $A$ precisely $T$ times around $S^1$, hence fixes $A$ pointwise.
$\psi^\tau$ is supported in $A$, hence $\phi^T$ and $\psi^\tau$ commute.

\begin{figure}[!htbp]
\begin{center}
\includegraphics[width=0.4\textwidth]{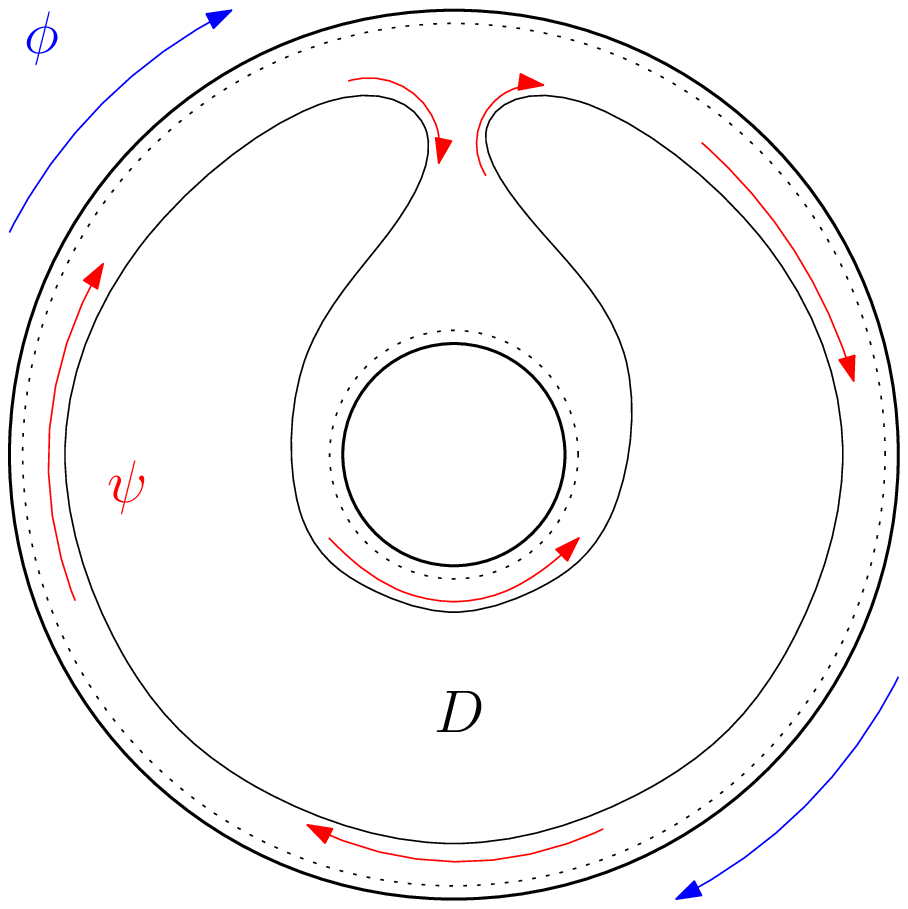}
\caption{$g_{T, \tau}$}
\end{center}
\end{figure}

\subsection{Proof of Theorem~\ref{T:annulus}.}

We claim that $g_{T, \tau}$ is not autonomous. Assume by contradiction that it is generated as the time-$1$ map of a Hamiltonian function $H : \Sigma \to \RR$. 
Suppose first that $H$ is generic, that is, $H$ admits a Reeb tree $T_H$.
We compute the values of $H$ at its percentile sets in two different ways: first, pick $h \in [0.01, 0.99]$. Let $a = 1$ and $b = 2h$ which satisfy $h = \frac{1+b-a}{2}$.
\[
  r_{a,b} (g_{T, \tau}) = r_{a,b} (\phi^T) + r_{a,b} (\psi^\tau) = h T.
\]
The first equality holds because $\phi^T$ and $\psi^\tau$ commute. $r_{a,b} (\phi^T) = h T$ since $Y_h = S^1 \times \{h\}$ is the $h$-percentile for $K$ and $K(Y_h) = h$.
$r_{a,b} (\psi^\tau) = 0$ as the support of $\Psi$ becomes displaceable in $S^2_{a,b}$.
Therefore, if the $h$-percentile $X_h$ exists for $H$, $H (X_h) = h T$.

We perform another computation: fix $h' \in [0.2, 0.8]$. Let $a' = 0.8 - h'$ and $b' = h' - 0.2$. Once again, $h' = \frac{1+b'-a'}{2}$ and
\[
  r_{a',b'} (g_{T, \tau}) = r_{a',b'} (\phi^T) + r_{a',b'} (\psi^\tau) = h' T + \tau.
\]
$r_{a',b'} (\phi^T) = h' T$ as before but $r_{a',b'} (\psi^\tau) = \tau$ since $i_{a', b'}$ embeds $\Sigma$ into a sphere of area $1 + a' + b' = 1.6$. So the image of the disk $D$ becomes the median set
for $i_{a', b',*} \Psi$, thus $r_{a',b'} (\psi^\tau)$ can be computed explicitly.
The calculation shows that if the $h'$-percentile $X_{h'}$ exists for $H$, $H (X_{h'}) = h' T + \tau$.

This contradicts the previous result, hence $h$-percentiles do not exist for $h$ in the interval $[0.2, 0.8]$. 
That is, $T_H$ has one or several branches with total measure at least $0.6$. In fact, there must be a single branch of measure at least $0.6$: 
if there are several branches growing out of different points of the stem, there will be intermediate $h$-percentiles which correspond to stem points between the branches. 
In our situation it is not the case. If there are two branches or more growing from the same stem point (which is possible in a non-generic situation), we may perturb $H$ in 
the $C^\infty$-topology and separate the branches. Intermediate percentiles will appear after such perturbation. However, our quasimorphisms $r_{a,b}$ are $C^0$-continuous, so a small perturbation
will not resolve the discrepancy $\tau$ between the results of two computations. 

As a corollary, there must be a branch $B \subset T_H$ with measure at least $0.6$. $D_B = H^{-1} (B)$ is a topological disk in $\Sigma$ of area at least $0.6$ which is an invariant set 
for the flow of $H$. Intuitively, points in $D_B$ have rotation number $0$ with respect to the $S^1$ coordinate
(all points with non-zero rotation number are mapped to the stem). However, most points in $\Sigma$ (up to a subset of area $0.02$) have 
rotation number $T$ under $g_{T, \tau}$, which gives a contradiction. 

\medskip

We reproduce this contradiction using more powerful tools. In ~\cite{Kh:h-spec}, Theorem 2, the author constructs a quasimorphism $\rho_{0.6} : \Ham (\Sigma) \to \RR$ which is $C^0$-continuous and has the following property.
Suppose $g \in \Ham (\Sigma)$ has an invariant disk of area $0.6$ or more, then $\rho_{0.6} (g)$ computes the rotation number (along the $S^1$ coordinate) of points in this disk.
($\rho_{0.6}$ is constructed as a certain combination of Calabi quasimorphisms pulled back from $S^2$ similarly to the construction of $r_{a,b}$.)
Therefore, 
\[
	\rho_{0.6} (g_{T, \tau}) = \rho_{0.6} (\phi^T) + \rho_{0.6} (\psi^\tau) = T.
\]
$\rho_{0.6} (\phi^T) = T$ since $\phi^T$ rotates the annulus $A$ $T$ times around, the same is true for any disk of area $0.6$ in $A$.
$\rho_{0.6} (\psi^\tau) = 0$ since $D$ is a stationary disk of area $0.8$.

This shows that large invariant disks of $g_{T, \tau}$ (if they exist) have rotation number $T$.
On the other hand, invariant disks of an autonomous flow must have rotation number zero. Therefore no such disks exist, 
so the Hamiltonian function $H$ cannot have a large branch. This is a contradiction to the first part of the 
argument where we established existence of a branch $B$. Hence $g_{T, \tau}$ is not autonomous.

If $H$ which is supposed to generate $g_{T, \tau}$ is extremely non-generic and its Reeb graph does not exist, we may perturb it and argue as before, since 
the quasimorphisms $r_{a,b}$ and $\rho_{0.6}$ used as tools to arrive to a contradiction are $C^0$-continuous.

\begin{rem}
  $g_{T, \tau}$ is not autonomous in $\Ham(\Sigma)$ but is a composition of two autonomous maps.
  
  However, if one allows Hamiltonian flows and diffeomorphisms in $\Sigma$ whose support is not restricted to the interior, $g_{T, \tau}$ 
  becomes autonomous in this extended group and, in particular, is integrable and has entropy zero.
  To see this, note that for an integer $T$, the map $\phi^T$ which rotates the inner subannulus $A$ around the $S^1$ coordinate, 
  can be generated by another autonomous flow $\widetilde{\phi}^t$. Namely, the one that fixes $A$ pointwise and rotates a tubular neighborhood of $\partial \Sigma$
  in the opposite direction. This flow is generated by $\widetilde{K}(\theta, s) = K (\theta, s) - s$.
  Since $\psi^t$ is supported inside $A$, $\widetilde{\phi}^t$ and $\psi^t$ have disjoint supports, commute and can be combined into a single autonomous flow which
  generates $g_{T, \tau} = \phi^T \circ \psi^\tau$.
\end{rem}
 
 The obstruction for $g_{T, \tau}$ to be autonomous consists of two ingredients:
  \begin{itemize}
	  \item
		 $\forall h \in [0.2, 0.8]. \, r_{a',b'} (g_{T, \tau}) -  r_{a,b} (g_{T, \tau}) = \tau \neq 0$, 
		 hence no percentiles exist in the interval $[0.2, 0.8]$. Therefore there must be a branch of area at least $0.6$.	  
	  \item
	   $\rho_{0.6} (g_{T, \tau}) = T \neq 0$.
	   Therefore $g_{T, \tau}$ cannot have a large invariant disk with rotation number zero, hence no large branches for 
	   the generating function of $g_{T, \tau}$.
  \end{itemize}

  The quasimorphisms $\rho$ and $r$ used in the argument are $C^0$-continuous and Hofer-Lipschitz, hence this obstruction persists under 
  deformations of $g_{T, \tau}$ that are either $C^0$-small or whose Hofer norm is less than $\min (T, \tau)$ divided by appropriate Lipschitz constants.
  This provides a lower bound for the Hofer distance between $g_{T, \tau}$ and the set of autonomous Hamiltonians. In particular,
  $g_{T, \tau}$ arrives arbitrarily far away from autonomous diffeomorphisms if we let $T, \tau \to \infty$.

\section{Proof of Theorem~\ref{T:main}} \label{S:General}

\subsection{Construction of the Hamiltonian $g$}

We consider the following diffeomorphism which is slightly different from the construction in Section~\ref{S:Annulus} but follows the same lines. 
Denote $A = S^1 \times (-2, 2)$. Let $(\theta, s)$ be the coordinates on $A$ 
(in our notation $S^1 = \RR / 2\pi \ZZ$) and equip $A$ with the area form $\dd\theta \wedge \dd s$ so that $A$ has symplectic area $8 \pi$.
Denote by $D = \{\theta^2 + s^2 < 1\}$ a unit disk in $A$.

Denote by $K : A \to \RR$ the Hamiltonian function $K (\theta, s) = s$. Its Hamiltonian vector field $X_K = \partial / \partial \theta$, 
so the Hamiltonian flow $\phi^t_K$ translates each point in $A$ by $t$ in the $S^1$ coordinate: $\phi^t_K (\theta, s) = (\theta + t, s)$.
Let $\varphi = \phi^\pi_K$ -- half rotation of the annulus. Note that $\varphi (D) \cap D = \emptyset$ while $\varphi^2 = \Id_A$. 
In particular, $\varphi^2$ fixes $D$ pointwise.

Let $\psi: D \to D$ be an autonomous diffeomorphism which rotates circles $\{\theta^2 + s^2 = r^2\}$ counterclockwise by an angle $\alpha(r)$.
We ask that $\alpha (r) > 0.2 $ in the inner disk $\{ r < 0.7\}$, equals zero for $r \geq 0.9$ 
and decreases monotonously with $r$. Clearly, $\psi$ is area-preserving. As it is compactly supported in $D$, $\psi$ extends to the annulus $A$ by identity.

Without loss of generality, we suppose that the area of $\Sigma$ is greater than $8 \pi$ so that $A$ admits a symplectic embedding $i : A \to \Sigma$ 
(it is not important whether $i (A)$ is contractible in $\Sigma$ or not). 
In what follows we identify $A$ with its image in $\Sigma$. We extend $K$ to the rest of $\Sigma$ in arbitrary way 
(for example, by constants, if $\Sigma$ is closed and $\Sigma \setminus A$ is disconnected, or by cutting $K$ off near $\partial A$). 
This induces an autonomous extension of the Hamiltonian vector field $X_K$ and its flow $\phi^t_K$ to $\Sigma$.
$\varphi = \phi^\pi_K : \Sigma \to \Sigma$ is an autonomous extension of $\varphi: A \to A$ which was discussed earlier.
$\psi$ extends outside $A$ by the identity map. Let $g = \varphi \circ \psi$. 

\begin{figure}[!htbp]
\begin{center}
\includegraphics[width=0.3\textwidth]{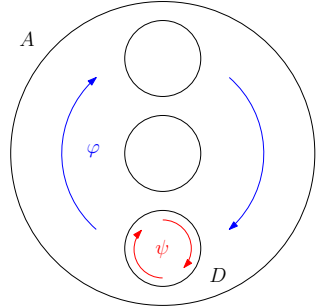}
\caption{The map $g$}
\end{center}
\end{figure}

\subsection{Proof of Theorem~\ref{T:main}.}

\begin{lem}
$g$ is integrable. In particular, it has entropy zero.
\end{lem}
\begin{proof}
The figure below provides a lamination which is preserved both by $\varphi$ and $\psi$, 
hence also by $g = \varphi \circ \psi$. This lamination extends outside the annulus $A$ to the rest of $\Sigma$ by level sets of the 
autonomous Hamiltonian $K$ which generates $\varphi$. Clearly, this extended lamination is also invariant under $\varphi \circ \psi$.
It is not difficult to construct a function which generates this lamination by level sets.

\begin{figure}[!htbp]
\begin{center}
\includegraphics[width=0.4\textwidth]{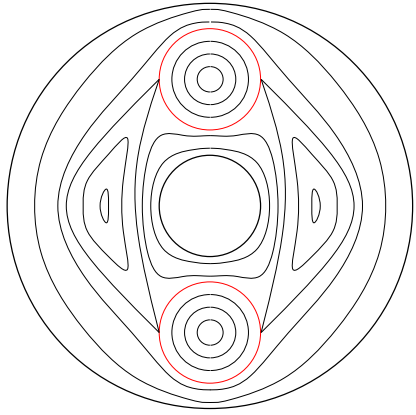}
\caption{Invariant lamination}
\end{center}
\end{figure}

\end{proof}

In our argument that $g$ cannot be approximated by autonomous diffeomorphisms we will focus on the annulus $A$, the dynamics outside $A$ is not important for the proof.
Very roughly, if $g$ had been generated by an autonomous flow, this flow would have eventually displaced $D$ (since $g(D) \cap D = \emptyset$) 
then brought it back later on applying a rotation around the center of $D$ (indeed $g^2 (D) = D$ as a set and $g^2|_D = \psi$). 
It turns out that a given autonomous flow can perform either of these deformations but not both, and this discrepancy persists under $C^0$ perturbations.

\medskip

Since the angle $\alpha(r)$ of rotation by $\psi$ increases from zero near $\partial D$ and becomes greater than $0.2$ 
near the center of $D$, there is a radius $r_{0.2}$ which corresponds to points 
which are rotated precisely by $\alpha(r_{0.2}) = 0.2$. 

Since all Riemannian metrics of $M$ are comparable on $D$ and define there the same notion of convergence, in further computations we will 
assume without loss of generality that the Riemannian metric of $M$ restricts to $\dd s^2 + \dd \theta ^2$ on $D$.
Assume by contradiction that there exists an autonomous Hamiltonian $h$ 
which is $C^0$-close to $g$. Note that $d_{C^0} (h^2, g^2) \to 0$ as $d_{C^0}(h, g)\to 0$,
so we may assume that $d_{C^0} (h, g) < 0.1$ and $d_{C^0} (h^2, g^2)<0.1$. 
Let $\chi$ be an autonomous flow which generates $h$.
We note that, similar to $g$, $h$ displaces $D$. Hence $h$ has no fixed points inside the disk, which means no equilibrium points for $\chi$ in $D$ 
(zeroes of the underlying vector field).

\begin{lem}
  $h^2$ has a fixed point in $D$.
\end{lem}
\begin{proof} We construct a vector field $X \in \Vect (D)$ by $X (p) = h^2 (p) - p$. Note that the vector field $g^2(p) - p$ 
points counterclockwise on the circle $\{ s^2 + \theta^2 = r^2_{0.2} \}$. Since $h^2$ is close to $g^2$, $X$ also points counterclockwise on this circle. 
Hence the rotated vector field $e^{-i \frac{\pi}{2}} X$ points outwards the disk 
$\{ s^2 + \theta^2 \leq r^2_{0.2} \}$ and has at least one zero there by the Poincar\'{e}-Hopf theorem. 
This implies that $X$ has a zero as well, which is a fixed point for $h^2$.
\end{proof}


Note that fixed points $p$ of an autonomous diffeomorphism $h$ on a surface can be of two types. 
Either $p$ is an equilibrium point for the flow or the flow line $c$ to which $p$ belongs is a simple loop. 
Moreover, if the time-$t$ map of the flow traverses $c$ $n$ times, the same is true for any other point in $c$. Namely,
the curve $c$ is fixed by $h$ pointwise.

Pick a fixed point $p$ of $h^2$ in $D' = \{ s^2 + \theta^2 \leq r^2_{0.2} \}$ as in the lemma. 
In our case $h$ has no fixed points in $D$, hence $p$ is not an equilibrium point. Therefore $p$ belongs to a loop $c$ in $\Sigma$ which consists of fixed points of $h^2$.
Note that $\partial D' = \{ s^2 + \theta^2 = r^2_{0.2} \}$ is rotated by $g^2$ by the angle $0.2$. 
Since all $x \in \partial D'$ are displaced by $g^2$ by distance greater than $0.1$, $\partial D'$ cannot contain fixed points of $h^2$.
Therefore $c \cap \partial D' = \emptyset$, which implies $c \subset D'$.
But in this case the disk encircled by $c$ is invariant under the flow $\chi$ which generates both $h^2$ and $h$, hence $h$ cannot displace $D$.
This is a contradiction and the proof follows.

\section{Final remarks}

There are several interesting bi-invariant norms on the group $\Ham(\Sigma)$: Hofer norm ~\cite{Hof:TopProp}, autonomous norm, 
entropy-zero norm and integrable norm ~\cite{Bra-Ma:ent-quas}. 
While it is easy to see that the last three norms are not equivalent to the Hofer norm, 
it is not known whether these norms are inequivalent.

Let us recall the definition of these norms. It is known that every Hamiltonian diffeomorphism of $\Sigma$ can be written as a composition
of autonomous diffeomorphisms and hence of entropy-zero and integrable diffeomorphisms as well. Let $Aut_{\Sigma}$, $Int_\Sigma$ and 
$Ent_\Sigma$ denote the sets of the autonomous, integrable and entropy-zero diffeomorphisms respectively. Since these sets are conjugation invariant,
the word norms: the autonomous norm $\|\cdot\|_{Aut}$, the integrable norm $\|\cdot\|_{Int}$ and the entropy-zero norm $\|\cdot\|_{Ent}$ 
(whose generating sets are $Aut_\Sigma$, $Int_\Sigma$ and $Ent_\Sigma$, respectively) on $\Ham(\Sigma)$ are bi-invariant. Note that 
$Aut_\Sigma\subset Int_\Sigma\subset Ent_\Sigma$, hence for every $f\in\Ham(\Sigma)$ we have 
$$\|f\|_{Ent}\leq \|f\|_{Int}\leq \|f\|_{Aut}\thinspace .$$

\begin{qsnonum}
Are the autonomous, integrable and entropy-zero norms equivalent?
\end{qsnonum}

\begin{rem}
The usual strategy used to try to attack the above question is to construct a quasimorphism which vanishes on one of the generating sets, 
but does not vanish on the other. In particular, an existence of a homogeneous quasimorphism which vanishes on the set $Int_\Sigma$, 
but does not vanish on the set $Ent_\Sigma$ implies that the entropy-zero and integrable norms are not equivalent. 
If, in addition, such quasimorphism is $C^0$-continuous, then this would imply a negative answer to the original question of Katok.
\end{rem}  

We finish the paper with a question that concerns the Hofer norm. In ~\cite{Kh:non-auton-curves} the second author showed that
in case when $\Sigma$ is an annulus, the diffeomorphism constructed in Section~\ref{S:Annulus} can not be presented as a Hofer limit of autonomous Hamiltonians.

\begin{qsnonum}
Let $\Sigma$ be a general surface. Is it true that a diffeomorphism $g$ of $\Sigma$, 
constructed in the general proof of Theorem ~\ref{T:main}, cannot be presented as a Hofer limit of autonomous Hamiltonian diffeomorphisms? 
\end{qsnonum}

\bibliography{bibliography}

\end{document}